\documentclass{amsart}
\usepackage{amssymb,euscript,amsmath, mathrsfs, amscd}
\usepackage[dvips]{graphicx}
\usepackage[dvips]{color}

\newcounter{ENUM}
\newcommand{\itm}{\item}
\newenvironment{ilist}{\renewcommand{\theENUM}{\roman{ENUM}}\renewcommand{\itm}{\addtocounter{ENUM}{1}\item[(\theENUM)]}\begin{itemize}\setcounter{ENUM}{0}}{\end{itemize}}

\newcommand{\margh}[1]{}

\def\risom{\overset{\sim}{\rightarrow}}

\input xy
\xyoption{all}
\CompileMatrices

\def\cG{{\mathcal G}}
\def\cAG{{\mathcal A}{\mathcal G}}

\def\cU{{\mathcal U}}

\def\cM{{\mathcal M}}
\def\cX{{\mathcal X}}
\def\cY{{\mathcal Y}}
\def\cZ{{\mathcal Z}}

\def\sE{{\mathscr E}}
\def\sF{{\mathscr F}}
\def\sG{{\mathscr G}}
\def\sK{{\mathscr K}}
\def\sL{{\mathscr L}}
\def\sM{{\mathscr M}}
\def\sO{{\mathscr O}}
\def\sV{{\mathscr V}}

\def\vp{\varphi}

\def\ch{\operatorname{char}}

\def\Hom{\operatorname{Hom}}

\def\Spec{\operatorname{Spec}}
\def\Pic{\operatorname{Pic}}

\def\res{\operatorname{res}}

\def\id{\operatorname{id}}

\def\codim{\operatorname{codim}}

\def\spn{\operatorname{span}}

\newtheorem{thm}{Theorem}[section]
\newtheorem{prop}[thm]{Proposition}
\newtheorem{lem}[thm]{Lemma}
\newtheorem{cor}[thm]{Corollary}

\theoremstyle{definition}
\newtheorem{defn}[thm]{Definition}

\newtheorem{ex}[thm]{Example}

\theoremstyle{remark}
\newtheorem{notn}[thm]{Notation}
\newtheorem{rem}[thm]{Remark}

\numberwithin{equation}{section}
\numberwithin{figure}{section}

\begin{document}
\title{Brill-Noether loci with fixed determinant in rank $2$}
\author{Brian Osserman}
\begin{abstract} In the 1990's, Bertram, Feinberg and Mukai examined
Brill-Noether loci for vector bundles of rank $2$ with fixed canonical
determinant, noting that the dimension was always bigger in this case
than the naive expectation. We generalize their results to treat a much
broader range of fixed-determinant Brill-Noether loci. The main technique
is a careful study of symplectic Grassmannians and related concepts.
\end{abstract}

\thanks{This work was primarily carried out during the 2009 special 
semester on Algebraic Geometry at MSRI, Berkeley.}
\maketitle

\section{Introduction}

Classical Brill-Noether theory studies the space of line bundles of
a given degree with a given number of global sections on a general 
smooth projective curve of given genus. The basic questions on 
non-emptiness, dimension, and smoothness were finally settled by
work of Kempf, Kleiman-Laksov, Griffiths-Harris, and Gieseker concluding
in the early 1980's.
The natural generalization of these questions to higher-rank vector
bundles remains very much open, even in the case of rank $2$, despite 
much progress due to many people (see \cite{g-t1} for a survey).
Complications (in the form of components of larger than the expected
dimensions) are known to occur due to instability of the underlying
bundle, and due to symmetries introduced by special determinants.
The former is easily avoided by restricting to stable or semi-stable
loci, but the latter seems to be a fundamental phenomenon. Our aim is to 
study this issue systematically, and our results are substantial 
generalizations of the original work of Bertram and Feinberg \cite{b-f2} 
and Mukai \cite{mu2} which pioneered the subject by studying Brill-Noether 
loci for fixed canonical determinant.

If $C$ is a smooth projective curve of genus $g$, and we are given $r,d,k$,
the space $\cG_C(r,d,k)$ parametrizing pairs of a vector bundle of rank 
$r$ and degree $d$, with a $k$-dimensional space of global sections,
has an expected dimension 
$$\rho(r,d,k,g)=r^2(g-1)+1-k(k-d+r(g-1)),$$ 
directly generalizing the classical case $r=1$. However, for $r>1$ the space 
turns out to be quite difficult to describe. Unlike the case $r=1$, even for 
a general curve $C$ it can be empty when $\rho(r,d,k,g) \geq 0$ (with 
suitable stability restrictions), non-empty when $\rho(r,d ,k,g) < 0$, and 
can have multiple components of distinct dimensions.

One case which appears to
be better behaved is the case of rank $2$ with canonical determinant.
This was studied independently by Bertram and Feinberg \cite{b-f2}
and Mukai \cite{mu2}, who both showed that in fact every component of 
the space $\cG(2,\omega,k)$ of pairs with canonical determinant has 
dimension at least
$$\rho_{\omega}(k,g):=3g-3-\binom{k+1}{2}=\rho(2,2g-2,k,g)-g+\binom{k}{2}.$$
They also conjectured that this locus had good behavior analogous to the
classical rank-$1$ setting.

Our main results are the following:

\begin{thm}\label{thm:exp-dim-twist} Let $C$ be a smooth projective curve of 
genus $g$, and $\sL$ a line bundle of degree $d$ on $C$. Let $\delta \geq 0$ be 
the smallest degree of an effective divisor $\Delta$ such that 
$h^1(C,\sL(-\Delta))\geq 1$. Then every irreducible 
component of the stack $\cG(2,\sL,k)$ of pairs with determinant $\sL$ has 
dimension at least
$$\rho^1_{\sL}(k,g):=\rho(2,d,k,g)-g+\binom{k-\delta}{2}.$$

Moreover, every pair $(\sE,V)$ corresponding to a point of $\cG(2,\sL,k)$
satisfies the condition 
$$\dim (V \cap H^0(C,\sE(-\Delta))) \geq k-\delta.$$
\end{thm}

See \S \ref{sec:prelims} for more details on definitions and notation.
Note that this estimate only differs from the usual one if $k \geq \delta+2$,
and the Bertram-Feinberg-Mukai result is simply the case of 
Theorem \ref{thm:exp-dim-twist} where $\sL=\omega$ (and therefore $\delta=0$).
More generally, we have:

\begin{cor} Let $C$ be a smooth projective curve of genus 
$g$, and $\sL$ a line bundle of degree $d$ on $C$, and suppose that
$h^1(C,\sL) \geq 1$. Then every irreducible component of $\cG(2,\sL,k)$ has 
dimension at least
$$\rho^1_{\sL}(k,g)=\rho(2,d,k,g)-g+\binom{k}{2}.$$
\end{cor}

The case that $\delta>0$ is also interesting, as it yields examples where
the dimension is larger than expected even in degree strictly greater than 
$2g-2$. This is perhaps not surprising if one thinks in terms of 
constructions via extensions, but it underlines the point that the sense
in which Brill-Noether theory in rank $2$ will depend on ``rank-$1$
phenomena'' goes beyond cases like canonical determinant, when the
determinant is special in the Brill-Noether sense, to include at least
some cases where the determinant is not special, but can be expressed as 
the tensor product of two special line bundles.

In the case $\delta=0$, we are further able to show that we have the 
beginning of a possible trend:

\begin{thm}\label{thm:exp-dim-double} Let $C$ be a smooth projective curve of 
genus $g$, and $\sL$ a line bundle of degree $d$ on $C$. Suppose 
that $h^1(C,\sL) \geq 2$. Then every irreducible 
component of $\cG(2,\sL,k)$ has dimension at least
$$\rho^2_{\sL}(k,g):=\rho(2,d,k,g)-g+2\binom{k}{2}.$$
\end{thm}

In particular, we find a range of candidates for components
of $\cG_C(2,\sL,k)$ having dimension strictly greater than 
$\rho(2,d,k,g)$.

\begin{cor}\label{cor:new-comps} Suppose we have $C,\sL,k$ as in 
Theorems \ref{thm:exp-dim-twist} and \ref{thm:exp-dim-double},
and further we either have $\delta$ as in Theorem \ref{thm:exp-dim-twist} 
such that
$\binom{k-\delta}{2}> g$,
or we have $h^1(C,\sL)>1$, with
$2\binom{k}{2}> g$,

Then if $\cG(2,\sL,k)$ is non-empty, it has dimension strictly greater than
$\rho(2,d,k,g)$.
\end{cor}

Even in cases where these modified expected dimensions turn out to be
sharp, one is left to determine when they are non-empty: certainly there
are cases where the expected dimension is positive but the corresponding
locus is empty.

Our techniques generalize the arguments of Bertram, Feinberg and Mukai
relating to symplectic forms on vector bundles. Theorem 
\ref{thm:exp-dim-twist} is a relatively straightforward generalization,
but Theorem \ref{thm:exp-dim-double} requires a detailed analysis of
the behavior of ``doubly symplectic Grassmannians,'' which parametrize
subbundles which are simultaneously isotropic for a pair of symplectic
forms. These Grassmannians appear to be new, and the analysis of their
singularities forms the technical core of the present paper.

One aspect of our techniques is that they transparently generalize to
the case where the curve $C$ is allowed to vary in families. We do not
state our results in this generality because we expect the main
applications of such generalization will arise when combined with the
appropriate degeneration techniques. 

\subsection*{Acknowledgements}

I would like to thank Izzet Coskun, Montserrat Teixidor i Bigas, and 
Peter Newstead for helpful conversations.

\section{Preliminaries}\label{sec:prelims}

We work throughout over a base field $F$, of arbitrary characteristic.
For convenience, we assume throughout that $F$ is algebraically
closed. However, Theorems \ref{thm:exp-dim-twist} and 
\ref{thm:exp-dim-double} hold for arbitrary fields, and indeed the general
case reduces immediately to the algebraically closed case.

Given a vector bundle $\sE$ on a scheme $X$, recall that an alternating
form on $\sE$ is a morphism $\langle,\rangle: \sE \otimes \sE \to \sO_X$
satisfying $\langle s, s\rangle = 0$ for any section $s$ of $\sE$ on an
open subset of $X$. This implies that $\langle,\rangle$ is skew-symmetric
and indeed the two are equivalent away from characteristic $2$. An 
alternating form $\langle,\rangle$ is symplectic if it is nondegenerate
on all fibers of $\sE$.

We argue extensively in terms of codimension. If $Z \subseteq X$ is
closed, but $Z$ and $X$ are not irreducible, recall that the standard 
codimension convention is that
$$\codim_X Z 
= \min_{Z' \subseteq Z} (\max_{X' \subseteq X} (\codim_{X'} Z')),$$
where $Z'$ and $X'$ range over the irreducible components of $Z$ and $X$,
respectively. Thus, we see that the condition that $Z$ has codimension at 
most $d$ in $X$ is weaker than the condition that every component of $Z$ 
has codimension at most $d$ in $X$, but this latter is equivalent to saying 
that every component of $Z$ has codimension at most $d$ in every component 
of $X$. This distinction is important in applications, and in our results 
(see for instance Corollaries \ref{cor:iso-intersect} and 
\ref{cor:iso-intersect-double}), the latter sort of bound will arise 
naturally. In addition, for lack of a suitable reference we have included
in Appendix \ref{sec:codim-stacks} a brief treatment of codimension for
algebraic stacks.

We work systematically with algebraic stacks to avoid difficulties and
pathologies arising from the use of coarse moduli spaces. However, our
arguments adapt to the corresponding coarse moduli spaces as long as
one is willing to restrict to the stable locus (and moreover, the dimension
statement for the stacks imply the same statements for the coarse moduli
spaces). We generally use calligraphic letters for stacks and roman
letters for schemes. 

Our notation for the various moduli spaces will be as follows.

\begin{defn} If $X$ is a scheme, and $\sE$ a vector bundle of rank
$r$ on $X$, then for any $k<r$ we denote by
$G(k,\sE)$ the relative Grassmannian scheme over $X$, with
$T$-valued points of $G(k,\sE)$ corresponding to rank-$k$ subbundles
of the pullback of $\sE$ from $X$ to $T$. Similarly, if $\cX$ is an 
algebraic stack, then $\cG(k,\sE)$ is the relative Grassmannian of
rank-$k$ subbundles of $\sE$. 
\end{defn}

The usual argument goes through in the case of algebraic stacks to
show that $\cG(k,\sE)$ is smooth of relative dimension $k(r-k)$ over
$\cX$.

Our convention throughout is that a curve is geometrically integral,
and $C$ always denote a smooth, projective curve of genus $g$ over $F$.
The canonical line bundle on $C$ is denoted by $\omega$.

\begin{defn} Let $\sL$ be a line bundle on $C$. For any $r \geq 2$, we have
the moduli stack $\cM(r,\sL)$ of vector bundles $\sE$ of rank $r$ on $C$ 
together with an isomorphism $\det \sE \risom \sL$. 
\end{defn}

This is an Artin stack, smooth over $F$ of dimension $(r^2-1)(g-1)$. On the 
open substack of stable bundles, it is a Deligne-Mumford stack (assuming 
that $\ch F \neq r$); note that the isomorphism $\det \sE \risom \sL$ reduces
the dimension of all automorphism groups by $1$.

\begin{defn} Given $C$, $\sL$ as above, and $k> 0$, we denote by
$\cG(r,\sL,k)$ the stack whose $T$-valued points consist of triples
$(\sE,\vp,\sV)$, where: 
\begin{ilist}
\itm $\sE$ is a vector bundle of rank $r$ on $C_T:=T \times_F C$;
\itm $\vp:\det \sE \risom \sL$ is an isomorphism;
\itm $\sV \subseteq p_{1*} \sE$ is locally free of rank $k$ on $T$,
and the inclusion into $p_{1*} \sE$ remains injective after arbitrary base 
change, in the sense that for all $T' \to T$, we have that the
induced map $\sV|_{T'} \to p_{1*}' (\sE|_{C_{T'}})$ remains injective, where
$p_{1}':C_{T'} \to T'$ is the projection morphism.
\end{ilist}
\end{defn}

As in the classical case, $\cG(r,\sL,k)$ may be constructed as a closed
substack of a relative Grassmannian stack over $\cM(r,\sL)$, by
twisting the universal bundle with (the pullback of) a sufficiently ample
divisor on $C$ prior to pushing forward. The only difference is that
because $\cM(r,\sL)$ is not quasicompact, one has to carry out the
construction locally on an open cover by quasicompact substacks.
We thus find that the map $\cG(r,\sL,k) \to \cM(r,\sL)$ is representable
by schemes, and is indeed projective, at least locally on $\cM(r,\sL)$.
The construction also shows that every component of $\cG(r,\sL,k)$ has
dimension at least $\rho-g$.

\section{Symplectic and Alternating Grassmannians}

Let $X$ be a scheme, and $\sE$ a vector bundle on $X$ of rank $r$.
Suppose we are given a symplectic form $\langle,\rangle$ on $\sE$. Then we 
make the following definitions:

\begin{defn} Given $k<r$, we have the {\bf symplectic Grassmannian}
$SG(k,\sE,\langle,\rangle)$ defined as the closed subscheme of
$G(k,\sE)$ consisting of subbundles isotropic for $\langle,\rangle$.
\end{defn}

Now suppose that $\langle,\rangle$ is an alternating but not necessarily
nondegenerate form on $\sE$.
We can similarly define:

\begin{defn} Given $k<r$, we have the {\bf alternating Grassmannian}
$AG(k,\sE,\langle,\rangle)$ defined as the closed subscheme of
$G(k,\sE)$ consisting of subbundles isotropic for $\langle,\rangle$.
\end{defn}

We make the following definition:

\begin{defn} We say that $\langle,\rangle$ has {\bf uniform degeneracy}
of rank $p$ if the kernel of the induced map
$\sE \to \sE^*$ is given by a subbundle $\sK$ of rank $p$, and furthermore
$\sK$ is equal to the kernel even after base change.
\end{defn}

Unless stated otherwise, we assume throughout this section that 
$\langle,\rangle$ has uniform degeneracy of rank $p$.
Note that $\langle,\rangle$ induces a symplectic form on $\sE/\sK$, so 
we necessarily have that $r-p=:2\delta$ is even. It is then standard that:

\begin{prop}\label{prop:alt-grass} $AG(k,\sE,\langle,\rangle)$ is non-empty
if and only if $k \leq r-\delta$, and in this case is universally 
topologically flat over $X$, with fibers which are pure of dimension
$$k(r-k)-\binom{k}{2}+\binom{k-\delta}{2}.$$
Furthermore, if a subspace $V \subseteq \sE|_x$ corresponds to a point of 
$AG(k,\sE,\langle,\rangle)$ lying over some $x \in X$, then
$$\dim V \cap \sK|_x \geq k-\delta.$$

In addition, $AG(k,\sE,\langle,\rangle)=SG(k,\sE,\langle,\rangle)$ is 
smooth in the case $p=0$.
\end{prop}

Note that if $p=0$, we necessarily have $k \leq \delta=\frac{r}{2}$ in order 
for $SG(k,\sE,\langle,\rangle)$ to be non-empty. Here, by ``universally
topologically flat'' we mean that every irreducible component of 
$AG(k,\sE,\langle,\rangle)$ dominates an irreducible component of $X$,
and the same holds after arbitrary base change.
Note that in particular any flat morphism is universally topologically
flat.

\begin{proof} We define a stratification of $AG(k,\sE,\langle,\rangle)$ 
by locally closed subschemes $T_i$ as follows. 
If $\sF$ is the universal subbundle on $AG(k,\sE,\langle,\rangle)$, then
$T_i$ is defined as the locus on which the rank of the intersection
$\sF \cap \sK$ is equal to $i$. We have $0 \leq i \leq p$ by definition,
so if we prove that each $T_i$ is smooth over $X$, the desired smoothness
statement for $SG(k,\sE,\langle,\rangle)$ will also follow. In fact, we will
prove that each $T_i$ is smooth over $X$ of relative dimension
$$i(p-i)+(k-i)(r-k)-\binom{k-i}{2},$$ 
non-empty if and only if
$$k-\delta \leq i \leq p,$$ 
with irreducible fibers, and such that each $T_i$ is in the closure of 
$T_{i-1}$ as long as $T_{i-1}$ is non-empty.

Let $\tilde{T}_i$ be the relative flag variety over $T_i$ parametrizing 
$\sF$ together with a complete flag of subbundles of $\sF/(\sF \cap \sK)$. 
Then $\tilde{T}_i$ is smooth over $T_i$ of relative dimension 
$\binom{k-i}{2}$.  On the other hand, we can construct $\tilde{T}_i$ as a 
tower of smooth schemes over $X$ as follows: we let $\tilde{T}_i^0$
be the relative Grassmannian of rank-$i$ subbundles $\sF^0$ of $\sK$. We 
then define each $\tilde{T}_i^j$ for $j=1,\dots,k-i$ over 
$\tilde{T}_i^{j-1}$ to be the open subscheme of the relative Grassmannian
parametrizing rank-$(i+j)$ isotropic subbundles $\sF^j$ of $\sE$ which 
contain $\sF^{j-1}$ and such that $\sF^j \cap \sK = \sF^0$. It is then
clear that $\tilde{T}_i^{k-i}=\tilde{T}_i$. We note that
since we are working with an alternating form
the condition that $\sF^j$ is isotropic is equivalent to 
the conditions that it contain $\sF^{j-1}$ and be contained in 
$(\sF^{j-1})^{\perp}$. Moreover, since the form is degenerate precisely
on $\sK$, and $\sF^{j-1} \cap \sK=\sF^0$, we see that
$(\sF^{j-1})^{\perp}$ has rank $r-j+1$. 
We thus see that if $\tilde{T}_i$ is non-empty, it is smooth over $X$ of 
relative dimension 
$$i(p-i)+\sum_{j=1}^{k-i} (r-i-2j+1) = i(p-i)+(k-i)(r-k).$$
Moreover, $\tilde{T}_i$ is non-empty if and only each $\sF^j$ may be chosen 
to have no larger than the desired intersection with $\sK$, which holds 
if and only if $p+(i+j) \leq (r-j+1)+i$ for all $j$, or equivalently 
$2i \geq p-r+2k -1=-2\delta+2k-1$, which is the same as $i \geq k-\delta$.

It remains to check that in every fiber, every point of $T_i$ is in the
closure of $T_{i-1}$ for $i-1 \geq k-\delta$.
Accordingly, suppose $X$ is a point, and $\sF_0$ is an isotropic 
subspace of $\sE$, with $\sF_0 \cap \sK$ having dimension $i$. Choose a
basis $e_1,\dots,e_k$ for $\sF_0$ such that $e_1,\dots,e_i$ span
$\sF_0 \cap \sK$. Let $e$ be a vector contained in 
$\sF_0^{\perp} \smallsetminus \spn(\sK,\sF_0)$; note that 
$\dim \sF_0^{\perp}=r-k+i=2\delta+p-k+i>p+k-i$ by hypothesis,
so such an $e$ exists. Then for $t \neq 0$, the vector space 
$\sF_t=\spn(e_1,\dots,e_{i-1},e_i+te,e_{i+1},\dots,e_k)$ lies in $T_{i-1}$,
and $\sF_t$ goes to $\sF_0$ as $t$ goes to $0$, so we conclude that 
$\sF_0$ is in the closure of $T_{i-1}$.

We thus have that $AG(k,\sE,\langle,\rangle)$ is universally topologically
flat, with pure fiber dimension
equal to the relative dimension of $T_{\max\{0,k-\delta\}}$, which is
$$\begin{cases}k(r-k)-k\delta+\binom{\delta+1}{2} : &k \geq \delta \\
k(r-k)-\binom{k}{2}:& k \leq \delta,\end{cases}$$
which yields the desired formula.
\end{proof}

We next recall the following well-known lemma in dimension theory.

\begin{lem}\label{lem:pullback-codim} Suppose $f:Y' \to Y$ is a morphism 
of $X$-schemes, locally of finite type.
Let $Z \subseteq Y$ be an irreducible closed subscheme of codimension $d$
containing a smooth point of $Y$ over $X$.
Suppose that either $X$ is regular, or $X$ is universally catenary and $Z$ 
is universally topologically flat 
over $X$. Then every irreducible component of $f^{-1}(Z)$ has codimension at 
most $d$.
\end{lem}

Recall that $X$ being universally catenary is a mild chain-theoretic
condition which is satisfied if, for instance, $X$ is of finite type
over a Cohen-Macaulay scheme (Corollary 6.3.7 of \cite{ega42}).

\begin{proof} Since codimension is local, we may suppose that $Y$ is
smooth over $X$. If $X$ is regular, then so is $Y$, and the result is a 
theorem of Hochster (Theorem 7.1 of \cite{ho1}),
based on a theorem of Serre. Suppose instead that 
$X$ is universally catenary, and $Z$ is universally topologically flat over 
$X$. We have that $f^{-1}(Z)$ is the 
intersection of the graph
$\Gamma_f \subseteq Y' \times_X Y$ with $Y' \times _X Z$. Moreover, 
$\Gamma_f$ is a local complete intersection in $Y' \times_X Y$ because
it is the pullback under $f \times \id$ of the diagonal 
$\Delta \subseteq Y \times_X Y$, which is a local complete intersection
subscheme by smoothness of $Y$ over $X$. Under our hypotheses, we have
$Y' \times_X Y$ locally Noetherian and catenary, so intersection with a 
local complete intersection can only decrease codimension of components,
and it is thus enough to see that
every component $Z'$ of $Y' \times_X Z$ has codimension at most $d$ in
every component of $Y' \times_X Y$ containing it. But we note that 
$Y' \times_X Z$ is topologically flat over $Y'$ by hypothesis, so $Z'$
is supported in a generic fiber over $Y'$, and we can compute codimension
in that fiber. Since codimension in a fiber is preserved under base change,
and both $Y$ and $Z$ have constant dimension in fibers by hypothesis,
we obtain the desired bound.
\end{proof}

\begin{rem} One sees that some additional hypotheses on $X$ or $Z$ along
the lines of those of the lemma are in fact necessary, as otherwise we can 
set $Y=X$, and let both $Y'$ and
$Z$ be closed subschemes, and we obtain subadditivity of intersection
codimension as a special case. That this is false without some
regularity hypotheses is well known. In our application, in fact both
sets of hypotheses of the lemma are satisfied.
\end{rem}

We finally conclude the desired statement on the dimension of spaces
of subbundles when degeneracy conditions are imposed with respect to
a symplectic form.

\begin{cor}\label{cor:iso-intersect} Suppose $\sE$ is a vector bundle of 
rank $r$ on an algebraic stack $\cX$ of finite type over a Cohen-Macaulay
algebraic stack. Given $\langle,\rangle$ a symplectic 
form on $\sE$, let $\sF$ and 
$\sG$ be subbundles of $\sE$ of ranks $s$ and $t$, with $\sG$ isotropic 
with respect to $\langle,\rangle$, and $\sF$ having degeneracy of uniform 
rank $s-2\delta$. Then for any $k>0$, every component of the closed substack
$\cG(k,\sF\cap\sG)$ of $\cG(k,\sE)$ parametrizing (necessarily 
isotropic) subbundles of $\sE$ contained in both $\sF$ and $\sG$ has 
codimension at most
$$k(2r-s-t)-\binom{k-\delta}{2}$$ 
in every component of $\cG(k,\sE)$. Moreover, every such subbundle 
intersects the degeneracy
subbundle of $\sF$ with rank at least $k-\delta$.
\end{cor}

Note that $\cG(k, \sF \cap \sG)$ is empty unless $k \leq t$, $s \leq r$, 
$2t \leq r$. One can check that these inequalities imply that the asserted 
codimension bound is in fact positive.

\begin{proof} 
We first prove the desired bound in the case that $\cX=X$ is a
(universally catenary) scheme.
Now, since every subbundle of $\sG$ is necessarily isotropic,
we can represent $G(k,\sF \cap \sG)$ as the locus of isotropic subbundles
of $\sF$ of rank $k$ which are contained in $\sG$. We thus start with
$AG(k,\sF,\langle,\rangle)$, which we claim has every component of 
codimension equal to 
$k(r-s)+\binom{k}{2}-\binom{k-\delta}{2}$
in $G(k,\sE)$. 
Indeed, by Proposition \ref{prop:alt-grass} we know every component of 
$AG(k,\sF,\langle,\rangle)$ is supported over a generic point of $X$,
so it is enough to carry out the computation in the generic fiber, and
the formula for the dimension of fibers then proves our claim.

Next, we note that $AG(k,\sF,\langle,\rangle)$ has a natural map to
$SG(k,\sE,\langle,\rangle)$ induced by $\sF \hookrightarrow \sE$, and
the pullback of $\sG$ to $AG(k,\sF,\langle,\rangle)$ also induces a 
natural map to $SG(t,\sE,\langle,\rangle)$. The locus $G(k,\sF\cap \sG)$
of interest is precisely the preimage in $AG(k,\sF,\langle,\rangle)$ of
the incidence correspondence 
$I \subseteq SG(k,\sE,\langle,\rangle)\times_X SG(t,\sE,\langle,\rangle)$.
The second projection makes $I$ into a $G(k,t)$-bundle over 
$SG(t,\sE,\langle,\rangle)$, so is smooth of relative dimension $k(t-k)$. On
the other hand, $SG(t,\sE,\langle,\rangle)$ is itself smooth of relative
dimension $t(r-t)-\binom{t}{2}$ over $X$, so we conclude that $I$ is smooth 
of relative dimension $k(t-k)+t(r-t)-\binom{t}{2}$ over $X$. Since 
$SG(k,\sE,\langle,\rangle)$ is smooth of relative dimension
$k(r-k)-\binom{k}{2}$, we have that $I$ has pure codimension
$k(r-t)-\binom{k}{2}$ in
$SG(k,\sE,\langle,\rangle)\times_X SG(t,\sE,\langle,\rangle)$.
Thus, by Lemma \ref{lem:pullback-codim} (using that $G(k,\sE)$ is catenary
to obtain additivity of codimension) we obtain the desired bound.

We now prove the general statement. Let $\cY$ be the Cohen-Macaulay 
algebraic stack over which $\cX$ has finite type, and
let $X \to \cX$ be a smooth cover of $\cX$ by a scheme $X$ which factors
through a smooth cover of $\cY$. Then $X$ is
necessarily universally catenary.
Let $\sE_X,\sF_X,\sG_X$ be the pullbacks of $\sE,\sF,\sG$ to $X$. Then
$\langle,\rangle$ pulls back to a symplectic form $\langle,\rangle_X$ on
$\sE_X$, and we have the identities 
$G(k,\sF_X\cap\sG_X)=\cG(k,\sF\cap\sG) \times_{\cX} X$ and
$G(k,\sE_X)=\cG(k,\sE) \times_{\cX} X$. Let $\cZ$ be an irreducible
component $\cG(k,\sF\cap\sG)$; by definition, we want to prove that
$Z=\cZ \times_{\cX} X$ has codimension at most
$k(2r-s-t)-\binom{k-\delta}{2}$ in
$G(k,\sE_X)$. But by Lemma \ref{lem:irred-cover}, every irreducible component
of $Z=\cZ \times_{\cX} X$ is an irreducible component of 
$G(k,\sF_X \cap\sG_X)$, so we conclude the desired codimension bound from 
the case of schemes.

Finally, the assertion that every point of $\cG(k,\sF\cap\sG)$ corresponds
to a bundle intersecting the degeneracy subbundle of $\sF$ with rank at 
least $k-\delta$ follows from Proposition \ref{prop:alt-grass}, 
since $\cG(k,\sF \cap \sG) \subseteq \cAG(k,\sF)$.
\end{proof}

\section{Multiply Symplectic Grassmannians}

Let $X$ be a scheme, and $\sE$ a vector bundle on $X$ of rank $r$.
Suppose we are given a collection
$$\underline{\langle,\rangle}=
\{\langle,\rangle_1, \dots, \langle, \rangle_m\}$$
of symplectic forms on $\sE$.
Then we make the following definition:

\begin{defn} Given $k<r$, we have the {\bf multiply symplectic Grassmannian}
$MSG(k,\sE,\underline{\langle,\rangle})$ defined as the closed 
subscheme of $G(k,\sE)$ determined by subbundles which are simultaneously
isotropic for every $\langle,\rangle_i \in \underline{\langle,\rangle}$. 
\end{defn}

Clearly, a subspace gives a point in 
$MSG(k,\sE,\underline{\langle,\rangle})$ if and only if it is isotropic
for an arbitrary linear combination of the $\langle,\rangle_i$.

A phenomenon which is not an issue in the classical case of a single form
is that the geometry of a multiply symplectic Grassmannian can depend 
radically on the relationship between the forms. In the classical case 
that $X$ is a point, if the forms are
sufficiently general, the multiply symplectic Grassmannian has
codimension $m\binom{k}{2}$ in $G(k,\sE)$. However, the case of interest
for us is not completely general, so we have to give a closer analysis.
In the case $m=2$, we find that the key condition for good behavior is
that as we vary over all linear combinations of the $\langle,\rangle_i$, 
the rank of the pairing on the subspace in question should not drop by more 
than $1$. More precisely, our main result is the following:

\begin{thm}\label{thm:dsg-crit} In the case $m=2$, suppose that $X$ is
regular. Given a field $K$ extending $F$, and a $K$-valued point $x$ of $X$,
a simultaneously isotropic subspace $V \subseteq \sE|_x$ corresponds to
a point of the doubly symplectic Grassmannian
$MSG(k,\sE,\underline{\langle,\rangle})$ which is smooth of
codimension $2\binom{k}{2}$ inside $G(k,\sE)$ if and only if for
every nonzero $K$-linear combination $\langle,\rangle$ of the
$\langle,\rangle_i|_x$,
there does not exist any $2$-dimensional subspace $V' \subseteq V$
for which the induced pairing
$$\langle,\rangle:V' \times \sE|_x/V \to F$$
vanishes identically.
\end{thm}

Because the proof of the theorem is rather technical, we first draw the
desired consequence for our applications.

\begin{cor}\label{cor:iso-intersect-double} Suppose $\sE$ is a vector bundle 
of rank $r$ on a regular algebraic stack $\cX$, and 
$\langle,\rangle_i$
for $i=1,2$ are symplectic forms on $\sE$. Let $\sF$ and 
$\sG$ be subbundles of $\sE$ of ranks $s$ and $t$, both isotropic with
respect to both of the $\langle,\rangle_i$. Let 
$\cG(k,\sF\cap\sG)$ denote the closed substack of $\cG(k,\sE)$ parametrizing 
subbundles of $\sE$ contained in both $\sF$ and $\sG$.  
Suppose that for field $K$ extending $F$, and some $K$-valued point $x$ of 
$\cX$, the subspace $\sG|_x \subseteq \sE|_x$
has the property that for all $2$-dimensional subspaces $V' \subseteq \sG|_x$,
the two
pairings $V' \times \sE|_x \to F$ induced by the $\langle,\rangle_i$ are 
linearly independent. Then for any $k$-dimensional subspace 
$V\subseteq \sF|_x \cap \sG|_x \subseteq \sE|_x$, every  
component of $\cG(k,\sF\cap\sG)$ passing
through the point corresponding to $V \subseteq \sE|_x$ has codimension at 
most
$$k(2r-s-t)-2\binom{k}{2}$$ 
in $\cG(k,\sE)$.
\end{cor}

\begin{proof} First consider the case that $\cX=X$ is a (regular) 
scheme.  We consider $G(k,\sF)$ as a closed subscheme of
$G(k,\sE)$, of codimension $k(r-s)$. We note that $G(k,\sF)$ has a natural
map to $MSG(k,\sE,\underline{\langle,\rangle}) \times_X 
MSG(t,\sE,\underline{\langle,\rangle})$,
with the map to $MSG(t,\sE,\underline{\langle,\rangle})$ induced by 
$\sG$ and the
structure map to $X$. Moreover, $G(k,\sF\cap\sG)$ is the preimage of
the incidence correspondence $I$ under this map. We claim that 
$MSG(k,\sE,\underline{\langle,\rangle}) \times_X 
MSG(t,\sE,\underline{\langle,\rangle})$ 
is smooth at the image of the point corresponding to $V \subseteq \sE|_x$,
and that $I$ has pure codimension $k(r-t)-2\binom{k}{2}$ at this point.
Smoothness in $MSG(t,\sE,\underline{\langle,\rangle})$ follows directly
from Theorem \ref{thm:dsg-crit} and our hypotheses. On the other hand,
because $V \subseteq \sG|_x$, any $2$-dimensional subspace of $V$ is
also contained in $\sG|_x$, so Theorem \ref{thm:dsg-crit} also implies
smoothness of $MSG(k,\sE,\underline{\langle,\rangle})$. The codimension
statement for $I$ is obtained by arguing as in Corollary
\ref{cor:iso-intersect}. We thus conclude the 
desired statement from Lemma \ref{lem:pullback-codim}. Again arguing
as in Corollary \ref{cor:iso-intersect}, we deduce the stack case from
the scheme case by passing to a smooth cover.
\end{proof}

The remainder of this section is devoted to proving Theorem 
\ref{thm:dsg-crit}, and will not be used elsewhere. The following lemma is 
used to reduce Theorem \ref{thm:dsg-crit} to the case that $X$ is a point.

\begin{lem}\label{lem:reduce-to-point} Suppose that $X$ is a regular
scheme, $f:Y \to X$ is smooth, and $Z=Z_1 \cap Z_2$ is the intersection
of closed subschemes $Z_i \subseteq Y$. Given $z \in Z$, suppose that
the $Z_i$ are Cohen-Macaulay at $z$, of pure codimension $d_i$ in $Y$, 
for $i=1,2$. If the image of $z$ in $X$ is $x$, we then have that 
$z$ is a smooth
point of $Z$, with $\codim_Y Z=d_1+d_2$ at $z$, if and only if 
$z$ is a smooth point of $Z_x$, with $\codim_{Y_x} Z_x=d_1+d_2$ at $z$.
\end{lem}

\begin{proof} Clearly if $Z$ is smooth over $X$ of the appropriate
codimension, then the same is true for the fiber $Z_x$. Conversely, 
by Theorem 17.5.1 of \cite{ega44}
if $Z_x$ is smooth of codimension $d_1+d_2$ at $z$, it is enough to show
that $Z$ is flat over $X$ at $z$. Because $Y$ is smooth over a regular
scheme, we have $\codim_Y Z \leq d_1+d_2$ at $z$, but because everything is
catenary, the bound on dimension in terms of fibers gives
$\codim_Y Z \geq \codim_{Y_x} Z_x = d_1+d_2$ at $z$, so we conclude that
$Z$ has the expected codimension in $Y$ at $z$, and it follows (see
for instance Lemma 4.5 of \cite{o-h1}) that $Z$ is Cohen-Macaulay at $z$, 
and then from the standard flatness criterion (Corollary 6.1.5 of
\cite{ega42}) that $Z$ is flat over $X$ at $z$.
\end{proof}

The idea of the proof of Theorem \ref{thm:dsg-crit} is to calculate the
tangent space as an intersection of the tangent spaces of the appropriate
symplectic Grassmannians. The first step is the following, which is a 
simple calculation, and left to the reader:

\begin{lem}\label{lem:sg-tangent} Let $E$ be a vector space over a field $K$ 
of dimension $r$, and $\langle,\rangle$ a symplectic form on $E$. Given 
$k\leq \frac{r}{2}$,
and a subspace $V \subseteq E$ of dimension $k$ which is isotropic for
$\langle,\rangle$, under the standard identification of $T_V(G(k,E))$
with $\Hom(V,E/V)$, we have that the tangent space 
$T_V(SG(k,E,\langle,\rangle))$ is identified with 
$$\{\vp \in \Hom(V,E/V): \forall v_1,v_2 \in V, 
\langle v_1, \vp(v_2)\rangle-\langle v_2,\vp(v_1)\rangle=0\}.$$
\end{lem}

Note that in the above, the terms $\langle v_1, \vp(v_2)\rangle$ and 
$\langle \vp(v_1),v_2\rangle$ are well-defined because $V$ is assumed
to be isotropic. Indeed, if $V$ is isotropic, then $\langle ,\rangle$
induces a homomorphism $\psi:E/V \to V^*$, and the lemma may be 
rephrased as saying that $\vp \in \Hom(V,E/V)$ is in the tangent space
to $SG(k,E,\langle,\rangle)$ if and only if $\psi \circ \vp$ is self-dual.
Because $\langle,\rangle$ is nondegenerate, as $\vp$ ranges through 
$\Hom(V,E/V)$, the composition $\psi \circ \vp$ ranges over all of
$\Hom(V,V^*)$. Choosing a suitable basis, one then sees that the condition 
that $\psi \circ \vp$ is symmetric imposes $\binom{k}{2}$ linear 
conditions on $\Hom(V,E/V)$, agreeing with our previous codimension 
calculation for symplectic Grassmannians. 

The main point is thus to prove:

\begin{prop}\label{prop:dsg-main} Given vector spaces $V,W$ of dimensions
$k \leq n$, and surjective linear maps
$\psi_1,\psi_2:W \to V^*$, the two systems of $\binom{k}{2}$ linear 
conditions imposed on $\vp \in \Hom(V,W)$ by the condition that 
$\psi_i \circ \vp$ be symmetric are dependent if and only if there
is a $2$-dimensional subspace $V' \subseteq V$ and scalars 
$\lambda_1,\lambda_2$ not both zero, such that for all $v \in V'$ and
$w \in W$, we have $\lambda_1 (\psi_1(w))(v)+\lambda_2(\psi_2(w))(v)=0$.
\end{prop}

One direction of the proposition is easy, and holds
for any number of linear maps (equivalently, any number of linear forms).
Before giving the relevant statement, we give an explicit description
of what it means to have a linear dependence among the $m \binom{k}{2}$
conditions in question.

\begin{lem}\label{lem:dependence-explicit} Given 
$\psi_1,\dots,\psi_m:W \to V^*$, and $v_1,\dots,v_k$ a basis of $V$,
the $m \binom{k}{2}$ conditions requiring that each $\psi_i \circ \vp$ be 
symmetric are linearly dependent if and only if there 
exist $c_{i,j,\ell}$ not all $0$,
for $i=1,\dots,m$ and $1 \leq j <\ell \leq k$, such that for all 
$\vp \in \Hom(V,W)$ we have
$$\sum_{i,j,\ell} 
c_{i,j,\ell}((\psi_i(\vp(v_j)))(v_{\ell})-(\psi_i(\vp(v_{\ell})))(v_j))=
0.$$
\end{lem}

\begin{proof} For each $i$ the condition that $\psi_i \circ \vp$ be 
symmetric is equivalent to requiring that for each $v,v' \in V$, we have
$(\psi_i(\vp(v)))(v')=(\psi_i(\vp(v')))(v)$. It is sufficient to
impose this condition for $v=v_j, v'=v_\ell$, with $j<\ell$, giving us
the $\binom{k}{2}$ conditions for each $i$. The desired assertion
follows immediately.
\end{proof}

We then have:

\begin{lem}\label{lem:msg-easy} Given vector spaces $V,W$ of dimensions
$k \leq n$, and surjective linear maps
$\psi_1,\dots,\psi_m:W \to V^*$, 
if there exists a $2$-dimensional subspace $V' \subseteq V$ and scalars 
$\lambda_1,\dots,\lambda_m$ not all zero, such that for all $v \in V'$ and
$w \in W$, we have $\sum_i \lambda_i (\psi_i(w))(v)=0$
then the $m$ systems of $\binom{k}{2}$ linear 
conditions imposed on $\vp \in \Hom(V,W)$ by the condition that 
$\psi_i \circ \vp$ be symmetric are dependent. 
\end{lem}

\begin{proof} Indeed, let $v_1,v_2$ be a basis for $V'$. The dependence
of conditions for each $\psi_i$ is explicitly constructed by the observation 
that for all $\vp$, we have
$$\sum_i \lambda_i(\psi_i(\vp(v_2))(v_1)-\psi_i(\vp(v_1))(v_2))=0.$$
Since $v_1,v_2$ are linearly independent, they can be extended to a basis
of $V$, and we obtained the desired linear dependence by Lemma
\ref{lem:dependence-explicit}.
\end{proof}

We now begin the core technical work for proving Proposition 
\ref{prop:dsg-main}. The following definition is convenient for our analysis.

\begin{defn} Given $k \leq n$, we say a $k \times n$ matrix $A$ is in
\textbf{generalized Jordan normal form} if the leftmost $k \times k$
block is in Jordan normal form. 
\end{defn}

We will always work with only a single matrix in generalized Jordan normal
form, so we can impose the following notation without fear of confusion:

\begin{notn} If $A$ is in generalized Jordan normal form, let $\lambda_i$
denote the diagonal entries of $A$, and set $\epsilon_i=(A)_{i,i+1}$ for 
$i=1,\dots,k-1$. We also use the convention that $\epsilon_0=\epsilon_k=0$.
\end{notn}

The first key technical lemma is the following.

\begin{lem}\label{lem:technical-first} Given $k \leq n$, let $A$ be a
$k \times n$ matrix, and suppose there exist scalars $c_{i,j},c'_{i,j}$
for all $1 \leq i<j \leq k$, not all zero, and such that for every 
$n \times k$ matrix $B$, we have the identity 
\begin{equation}\label{eq:cij-matrix} \sum_{i<j}c_{i,j}(B_{i,j}-B_{j,i})=
\sum_{i<j}c'_{i,j}((AB)_{i,j}-(AB)_{j,i}).
\end{equation}
Suppose further that $A$ is in generalized Jordan normal form, and
$c'_{i,j} \neq 0$ for some $i<j$ occurring in the generalized
Jordan blocks in the $i_0,\dots,i_1$ and $j_0,\dots,j_1$ rows of $A$,
with $i_0 \leq i_1,j_0 \leq j_1$. Then we have:
\begin{ilist}
\itm $i_1 < j_0$;
\itm $\lambda_i=\lambda_j$ for any $i,j$ with $i_0 \leq i \leq i_1$ and
$j_0 \leq j \leq j_1$;
\itm the matrix $(c'_{i,j})_{i_0 \leq i \leq i_1, j_0 \leq j \leq j_1}$
takes constant values along antidiagonals, with non-zero entries only
occurring for $i+j \geq \max\{i_0+j_1,i_1+j_0\}$. 
\end{ilist}
\end{lem}

\begin{proof} We use the convention that $c_{i,i}=c'_{i,i}=0$ for 
$i=1,\dots,k$.
Given any $i,j$ with $i+1<j$, we first observe that if
$B_{i+1,j}=1, B_{j,i+1}=b$, 
and $B_{\ell,\ell'}=0$ for all other $\ell,\ell'$, then
\eqref{eq:cij-matrix} simplifies to
$$c_{i+1,j}(1-b)=c'_{i+1,j}(\lambda_{i+1}-b\lambda_j)
+c'_{i,j}\epsilon_{i}-c'_{i+1,j-1}b\epsilon_{j-1}.$$
Since this holds for all $b$, we find 
\begin{equation}\label{eq:mess}
c_{i+1,j}=c'_{i+1,j}\lambda_{i+1}+c'_{i,j}\epsilon_i
=c'_{i+1,j}\lambda_j+c'_{i+1,j-1}\epsilon_{j-1}.
\end{equation}
Also obersve that if $i+1=j$, setting $B_{i+1,i+1}=1$ and all other
entries equal to $0$ similarly gives 
\begin{equation}\label{eq:no-mess}
0=c'_{i,i+1}\epsilon_i.
\end{equation}

We will prove by induction on $j-i$ that if the $i$th and $j$th rows 
are in the same Jordan block of $A$, then $c'_{i,j}=0$. Being in the same
block is equivalent to $\epsilon_i=\epsilon_{i+1}=\dots=\epsilon_{j-1}=1$,
and then $\lambda_i=\lambda_{i+1}=\dots=\lambda_j$, so \eqref{eq:mess} 
reduces to
\begin{equation}\label{eq:less-mess}
c_{i+1,j}=c'_{i+1,j}\lambda_i+c'_{i,j}
=c'_{i+1,j}\lambda_i+c'_{i+1,j-1}.
\end{equation}
The base cases are $j=i+1$ and $j=i+2$. In the first case, 
\eqref{eq:no-mess} together with $\epsilon_i=0$ immediately yields
$c'_{i,i+1}=0$, as desired. In the case $j=i+2$, \eqref{eq:less-mess}
yields $c'_{i,i+2}=c'_{i+1,i+1}=0$ by our earlier convention.
Now we consider the case of arbitrary $i<j-1$. 
Then \eqref{eq:less-mess} yields $c'_{i,j}=c'_{i+1,j-1}$, so by the
induction hypothesis, we have $c'_{i,j}=0$, as desired. We thus conclude
(i).

To prove (ii), suppose we are given $i_0<i_1<j_0<j_1$ such that there 
exist $i,j$
with $c'_{i,j}\neq 0$, $i_0\leq i \leq i_1, j_0 \leq j \leq j_1$, and
further suppose we have chosen $i,j$ so that $i+j$ is minimal with respect
to these conditions. Thus we have either $c'_{i-1,j}=0$, or $i=i_0$, so
that $\epsilon_{i-1}=0$, and similarly either $c'_{i,j-1}=0$ or 
$\epsilon_{j-1}=0$. Replacing $i+1$ with $i$ in \eqref{eq:mess}, we find 
$c'_{i,j}(\lambda_i-\lambda_j)=c'_{i,j-1}\epsilon_{j-1}
-c'_{i-1,j}\epsilon_{i-1}=0$, so we conclude that $\lambda_i=\lambda_j$.

Finally, (iii) is equivalent to the assertions that when $c'_{i,j}\neq 0$,
if $\epsilon_i = 1$ then $j>j_0$ and $c'_{i+1,j-1}=c'_{i,j}$, and if 
$\epsilon_j = 1$ then $i>i_0$ and $c'_{i-1,j+1}=c'_{i,j}$. First suppose
$\epsilon_i=1$. Then $i<i_1<j$, so from \eqref{eq:mess} we obtain
$$c'_{i,j}\epsilon_i-c'_{i+1,j-1}\epsilon_{j-1}
=c'_{i+1,j}(\lambda_j- \lambda_{i+1}).$$
We also have $\lambda_{i+1}=\lambda_i=\lambda_j$
by (ii), so the right hand side is $0$. Since $c'_{i,j}\epsilon_i=c'_{i,j}$
is assumed non-zero, we conclude that $\epsilon_{j-1} \neq 0$ and hence
$j>j_0$ and $c'_{i+1,j-1}=c'_{i,j}$, as desired. The case that 
$\epsilon_j=1$ is obtained in the same way, if we replace $i+1$ by $i$ 
and $j-1$ by $j$ in \eqref{eq:mess}.
\end{proof}

We can now prove the next lemma, which yields Proposition \ref{prop:dsg-main}
almost immediately.

\begin{lem}\label{lem:technical-second} In the situation of Lemma 
\ref{lem:technical-first}, with $c'_{i,j} \neq 0$ for some $i<j$, so that
$\lambda_i=\lambda_j$, we have
$$\dim \ker (A-\lambda_i I_{k,n})^T \geq 2.$$
\end{lem}

\begin{proof} Continuing with the convention that $c'_{i,i}=0$, we also
extend $c'_{i,j}$ for $j<i$ by setting $c'_{i,j}=-c'_{j,i}$.
Suppose we choose $i$ minimal in a given Jordan block of $A$
such that $c'_{i,j} \neq 0$ for some $j$, and set 
$v_i=(c'_{i,1},\dots,c'_{i,k})$. Then we claim that 
$v_i (A -\lambda_i I_{k,n})=0$. To see this, we first observe that for
any $j$ with $c'_{i,j} \neq 0$, the minimality of $i$ within its Jordan 
block gives $c'_{i-1,j}\epsilon_{i-1}=0$. For $j>i$, replacing $i+1$ with 
$i$ in \eqref{eq:mess} we have
\begin{equation}\label{eq:eigenvalue}
c_{i,j}=c'_{i,j}\lambda_i+c'_{i-1,j}\epsilon_{i-1}=c'_{i,j}\lambda_i.
\end{equation}
Under our convention for $j<i$, one checks that we still obtain
\eqref{eq:eigenvalue} from \eqref{eq:mess}.
In order to verify our claim, we return to \eqref{eq:cij-matrix}, and
for any $j=1,\dots,n$, set $B_{j,i}=1$, and all other coefficients of
$B$ to be $0$. In this case, \eqref{eq:cij-matrix} yields
\begin{align*}-c_{i,j}
& = -\sum_{\ell>i}c'_{i,\ell}A_{\ell,j}+\sum_{\ell<i}c'_{\ell,i}A_{\ell,j}\\
& = - \sum_{\ell=1}^k c'_{i,\ell}A_{\ell,j},
\end{align*}
using the convention that $c_{i,j}=0$ if $j>k$.
The righthand side is equal to the dot product
$-v_i \cdot (A_{1,j},\dots,A_{k,j})$,
while our \eqref{eq:eigenvalue} gives us that the lefthand side is 
$-\lambda_i c'_{i,j}=
-v_i \cdot (\underbrace{0,\dots,0}_{j-1},\lambda_i,0,\dots,0)$. 
Since the identity holds for all $j=1,\dots,n$, we conclude that 
$v_i (A -\lambda_i I_{k,n})=0$, as desired.

It remains to prove that given any $\ell<\ell'$ with $c'_{\ell,\ell'} \neq 0$,
there are at least two choices of $i$ as above such that the corresponding 
$v_i$ are linearly independent. But this is clear: we have from 
Lemma \ref{lem:technical-first} that the $\ell$th and $\ell'$th rows are
contained in different Jordan blocks with the same eigenvalue, so
we can take $i$ minimal as above for each of the two Jordan blocks 
gives the desired pairs of linearly independent $v_i$.
\end{proof}

We can now quickly conclude our main results.

\begin{proof}[Proof of Proposition \ref{prop:dsg-main}]
One direction is proved by Lemma \ref{lem:msg-easy}. For the converse,
we suppose that we have a linear dependence among the $2\binom{k}{2}$
conditions imposed by the condition that both $\psi_i \circ \vp$ be
symmetric for all $\vp \in \Hom(V,W)$. It is clear that we can choose
bases of $V$ and $W$ so that (in terms of the induced dual basis
on $V^*$) we have that $\psi_1$ is the matrix consisting of the 
$k \times k$ identity matrix followed by the $k \times (n-k)$ zero matrix.
Moreover, since we have assumed that $F$ is algebraically closed, by 
modifying the bases compatibly, we can further put the
matrix $A$ corresponding to $\psi_2$ into generalized Jordan normal form. 
Translating Lemma \ref{lem:dependence-explicit} into
this context, we find that the existence of a linear dependence is
precisely equivalent to the existence of $c_{i,j},c'_{i,j}$ as in
\eqref{eq:cij-matrix}. According to Lemma \ref{lem:technical-second},
we can then produce the desired space $V'$ as any $2$-dimensional
subspace of $\ker(A-\lambda_i I_{k,n})^T$, yielding the desired result. 
\end{proof}

\begin{proof}[Proof of Theorem \ref{thm:dsg-crit}] Because the
individual symplectic Grassmannians $SG(k,\sE,\langle,\rangle_i)$
for $i=1,2$ are smooth, we can apply Lemma \ref{lem:reduce-to-point}
to reduce the theorem to the case that $X$ is a point. In this case,
again using smoothness of the symplectic Grassmannians, we have that
$MSG(k,\sE,\underline{\langle,\rangle})$ is smooth of codimension
$2\binom{k}{2}$ at a point if and only if the tangent spaces of the
two symplectic Grassmannians intersect transversely at that point.
Lemma \ref{lem:sg-tangent} describes these tangent spaces, and 
applying Proposition \ref{prop:dsg-main} to the case that 
$W=\sE|_x/V$ and the $\psi_i$ are induced by the $\langle,\rangle_i$ 
then yields the desired statement.
\end{proof}

\section{Expected dimensions in fixed determinant}\label{sec:exp-dim}

We can now prove Theorems \ref{thm:exp-dim-twist} and 
\ref{thm:exp-dim-double}. The lemma introducing the required symplectic
forms is essentially due to Bertram-Feinberg and Mukai, stated slightly
more generally as follows:

\begin{lem}\label{lem:construct-form} Given $\sL$ of degree $d$ on $C$,
and an effective divisor $\Delta$ of degree $\delta$ on $C$,
suppose that $\vp:\sL \to \omega(\Delta)$ is a non-zero morphism. 
Let $\cX$ be a stack over $\Spec F$, and $\sE$ a vector bundle of rank 
$2$ on $C \times \cX$, with 
an isomorphism $\det \sE \risom p_1^* \sL$.
Then for any effective divisor $D$ on $C$, if we set $D'=p_1^* D$
on $C \times \cX$, we have that
$\widetilde{\sE}:=p_{2*}(\sE(D')/\sE(-D'-\Delta))$ is a vector bundle of
rank $4\deg D+2\delta$ on $\cX$, and we obtain an alternating form 
$\langle,\rangle$ on $\widetilde{\sE}$ defined by
$$\langle s_1,s_2 \rangle = 
\sum_{P \in C} 
\res_P \vp(\tilde{s}_{1,P} \wedge \tilde{s}_{2,P}),$$
where each $\tilde{s}_{i,P}$ is a representative in $\sE(D')$ of
$s_i$ in a suitable neighborhood of $P$, and we extend $\vp$ to a map
from rational sections of $\sL$ to rational sections of $\omega$.

If further both $D'$ and $\Delta$ have support disjoint from the vanishing 
locus of $\vp$, then $\langle,\rangle$ is a symplectic form,
and $p_{2*}(\sE/\sE(-D'-\Delta))$ is a subbundle of $\widetilde{\sE}$ 
of rank $2\deg D+2\delta$, with uniform degeneracy of rank $2\deg D$. If 
also for every geometric point $x \in \cX$, we have that the
corresponding fiber $\sE|_x$ of $\sE$ satisfies 
$$h^1(C,\sE|_x(D))=h^0(C,\sE|_x(-D-\Delta))=0,$$
then $p_{2*}(\sE(D'))$ is a
subbundle of $\widetilde{\sE}$ of rank $d+2\deg D+2-2g$, and is isotropic
for $\langle,\rangle$. Moreover, under the given hypotheses the 
construction of all bundles in question is compatible with base change.
\end{lem}

Of course, the importance of the lemma is that we can recover
$p_{2*}(\sE)$ as the intersection of $p_{2*}(\sE/\sE(-D'-\Delta))$ 
with $p_{2*}(\sE(D'))$ inside $\widetilde{\sE}$. 

\begin{proof} First, the asserted compatibility with base change
immediately reduces the lemma to the case of a base scheme.
Since $\sE(D')/\sE(-D'-\Delta)$ has fibers which
are skyscraper sheaves of length $4 \deg D + 2 \delta$, by cohomology
and base change we have that pushforward commutes with base change,
and $\widetilde{\sE}$ is locally free of rank $4 \deg D + 2 \delta$.
Similarly, it's clear that $p_{2*}(\sE/\sE(-D'-\Delta))$ is a subbundle 
of $\widetilde{\sE}$ of rank $2\deg D+2\delta$.

Next, note that $\vp(\tilde{s}_{1,P} \wedge \tilde{s}_{2,P})$
is a section of $\omega(2D+\Delta)$ in a neighborhood of $P$, and thus
the residue can only be non-zero if $P$ is in the support of $D$ or of
$\Delta$. If either
$\tilde{s}_{i,P}$ is in $\sE(-D'-\Delta)$, then we have
$\vp(\tilde{s}_{1,P} \wedge \tilde{s}_{2,P})$ a local section of
$\omega$, so the residue vanishes. We thus obtain a well-defined alternating
form on $\widetilde{\sE}$.

If $D$ and $\Delta$ are disjoint from the 
vanishing locus of $\vp$, a local calculation
shows that $\langle,\rangle$ is nondegenerate, hence
symplectic. Similarly, in this case the degeneracy subbundle of 
$p_{2*}(\sE/\sE(-D'-\Delta))$ is precisely 
$p_{2*}(\sE(-\Delta)/\sE(-D'-\Delta))$, so has rank $2\deg D$, as asserted.

Now, if $h^1(C,\sE|_x(D))=0$ for all $x$, then by cohomology and
base change pushforward commutes with base change, and $p_{2*}(\sE(D'))$ is
locally free, with its rank $d+2 \deg D + 2 - 2g$ given by the Riemann-Roch
formula. The hypothesis that $h^0(C,\sE|_x(-D-\Delta))=0$ for all $x$
insures that $p_{2*}(\sE(D'))$ injects into $\widetilde{\sE}$, and
continues to inject after arbitrary base change, which in particular implies
that it is a subbundle. Finally, it is
isotropic for $\langle,\rangle$ by the residue theorem, since in this 
case we have
every $\tilde{s}_{i,P}$ coming from a single pair of global sections,
so we are simply summing the residues of a single rational differential
form.
\end{proof}

\begin{proof}[Proof of Theorem \ref{thm:exp-dim-twist}] 
Because $h^1(C,\sL(-\Delta)) > 0$, we have 
$h^0(C,\sL^{-1}\otimes \omega(\Delta)) > 0$, and we can thus fix a non-zero
morphism $\vp:\sL \to \omega(\Delta)$. Since $\delta$ is assumed minimal,
we moreover have that the support of $\Delta$ is disjoint from the 
vanishing locus of $\vp$.

Let $\cU$ be any open substack of $\cM(2,\sL)$ of finite type,
and $\sE_\cU$ the universal bundle on $C\times \cU$. Then a sufficiently
general choice of sufficiently ample divisor $D$ on $C$ will be disjoint 
from the vanishing locus of $\vp$, and because we assumed $\cU$ is of finite
type, for every vector bundle $\sE|_x$ 
corresponding to a point $x \in \cU$ we will have 
$H^1(C,\sE|_x(D))=H^0(C,\sE|_x(-D-\Delta))=0$.
Fixing such a $D$, let $D'=p_1^*(D)$ on $C\times \cU$.
Now, according to Lemma \ref{lem:construct-form} we have that
$p_{2*} (\sE_{\cU}(D')/\sE_{\cU}(-D'-\Delta))$ is locally free of rank 
$4 \deg D + 2\delta$, with subbundles 
$p_{2*}(\sE_{\cU}/\sE_{\cU}(-D'-\Delta))$ and $p_{2*}\sE_{\cU}(D')$, of
rank $2\deg D+2\delta$ and $d+2\deg D +2-2g$ respectively. Moreover, 
$p_{2*} (\sE_{\cU}(D')/\sE_{\cU}(-D'-\Delta))$ carries a symplectic form 
$\langle,\rangle$, with $p_{2*}\sE_{\cU}(D')$ isotropic for 
$\langle,\rangle$, and $p_{2*}(\sE_{\cU}/\sE_{\cU}(-D'-\Delta))$ having 
uniform degeneracy of rank $2 \deg D$.

Let $\cG$ be the relative Grassmannian 
$\cG(k,p_{2*} (\sE_{\cU}(D')/\sE_{\cU}(-D'-\Delta)))$, which
is smooth of relative dimension $k(4 \deg D + 2 \delta- k)$ over $\cU$. 
Let $\sF$ be the universal subbundle on $\cG$, and set
$\cG_{\cU}(2,\sL,k)$
to be the open substack of $\cG(2,\sL,k)$ parametrizing pairs for
which the underlying bundle is contained in $\cU$. Then observe that
$\cG_{\cU}(2,\sL,k)$ can be identified with the closed substack of $\cG$ 
on which $\sF$ is contained in both 
$p_{2*}(\sE_{\cU}/\sE_{\cU}(-D'-\Delta))$ and $p_{2*}\sE_{\cU}(D')$.
Applying Corollary \ref{cor:iso-intersect}, we find
that every component of $\cG_{\cU}(2,\sL,k)$ has codimension in $\cG$ bounded 
above by $$k(4\deg D + 2\delta-d-2+2g)-\binom{k-\delta}{2}.$$
By Corollary \ref{cor:smooth-codim}, we thus find that every component of 
$\cG_{\cU}(2,\sL,k)$ has dimension at least
$$3g-3-k(k-d+2g-2)+\binom{k-\delta}{2}=\rho-g+\binom{k-\delta}{2}.$$

As $\cU$ varies over all bounded open substacks of $\cM(2,\sL)$, it covers
the entire moduli space, so we conclude that the obtained dimension bounds
apply to all of $\cG(2,\sL,k)$, as desired. The final assertion is an
immediate consequence of the last statement of Corollary 
\ref{cor:iso-intersect}, since the degeneracy subbundle of 
$p_{2*}(\sE_{\cU}/\sE_{\cU}(-D'-\Delta))$ is precisely 
$p_{2*}(\sE_{\cU}(-\Delta)/\sE_{\cU}(-D'-\Delta))$.
\end{proof}

We now move on to proving Theorem \ref{thm:exp-dim-double}. The lemma 
which allows us to apply Corollary \ref{cor:iso-intersect-double} is
the following, which holds more generally than the case that $h^1(C,\sL)=2$.

\begin{lem}\label{lem:general-m} Given $\sL$ of degree $d$ on $C$,
suppose that $h^1(C,\sL)=m>0$. Choose $m$ linearly independent morphisms
$\vp_i:\sL \to \omega$. Let $\cU$ be a finite-type open substack of
$\cM(2,\sL)$, and $\sE_{\cU}$ the universal bundle on $C \times \cU$.
Then for any sufficiently general choice of a sufficiently ample divisor
$D$ on $C$, if $D'=p_1^* D$, the alternating forms $\langle,\rangle_i$
on $\widetilde{\sE}:=p_{2*}(\sE(D')/\sE(-D'))$ induced by the $\vp_i$ and 
summing over residues as in Lemma \ref{lem:construct-form}
have the following property: for every point $x \in \cU$, and
$\lambda_1,\dots,\lambda_m$ not all $0$, the alternating form
$\sum_i \lambda_i \langle,\rangle_i$
induces an injective map $p_{2*}(\sE|_x(D)) \to (\widetilde{\sE}|_x)^*$.
\end{lem}

Note that in particular, the lemma implies that for $D$ sufficiently
general and sufficiently ample, the $\langle,\rangle_i$ are linearly 
independent.

\begin{proof} Since $D$ is assumed to be sufficiently ample, a sufficiently
general choice of $D$ will be a reduced divisor $P_1+\dots+P_{\deg D}$
for distinct points $P_i \in C$. Let $f_1=1,f_2,\dots,f_m$
be the rational functions on $C$ such that $\vp_i=f_i \vp_1$. Any 
sufficiently general $D$ will have support disjoint from the vanishing of
each $\vp_i$, and consequently from the zeroes and poles of the
$f_i$. Now, $\sE(D')/\sE(-D')$ has
support only at the $P_i$, of length $4$ at each point. Suppose we have
$s\in H^0(C,\sE|_x(D)/\sE|_x(-D))$ and constants $\lambda_i$
for $i=1,\dots,m$ such that for all $s' \in H^0(C,\sE|_x(D)/\sE|_x(-D))$,
we have 
\begin{equation}\label{eq:relation-forms} 
\sum_i \lambda_i \langle s,s'\rangle_i =0.
\end{equation}
Equivalently, \eqref{eq:relation-forms} is satisfied
for any $s'$ supported only at a single $P_j$. 
A local calculation at $P_j$ shows if $s$ is non-zero at 
$P_j$ (that is, if a local representative for $s$ is not in $\sE|_x(-D)$), 
a necessary condition for \eqref{eq:relation-forms} to hold for all
$s'$ supported only at $P_j$ is the identity
\begin{equation}\label{eq:relation-functions}
\sum_i \lambda_i f_i(P_j)=0.
\end{equation}

Now, because $\cU$ is quasicompact, the instability of the corresponding
vector bundles is bounded, and it follows that if $s$ is actually a
section in $H^0(C,\sE|_x(D))$, then the number of $P_j$ at which $s$
can vanish (in the sense that $s \in \sE|_x(-P_j)$ in a neighborhood of
$P_j$) grows more slowly than $\deg D$ as $\deg D$ increases,
and in particular for $D$ sufficiently ample, we have that any such $s$ 
must be non-zero at 
at least $m$ of the $P_j$. Then for $D$ sufficiently general,
the linear independence of the $f_i$ implies that 
\eqref{eq:relation-functions} cannot hold at all $m$ of the $P_j$ where
$s$ is non-zero.
We conclude the desired injectivity.
\end{proof}

\begin{proof}[Proof of Theorem \ref{thm:exp-dim-double}]
Because $h^1(C,\sL) \geq 2$, we have 
$h^0(C,\sL^{-1}\otimes \omega) \geq 2$, and we can thus fix a pair
of linearly independent morphisms $\vp_1,\vp_2:\sL \to \omega$. 

As before, let $\cU$ be any open substack of $\cM(2,\sL)$ of finite 
type, and $\sE_\cU$ the universal bundle on $C\times \cU$. Then a sufficiently
general choice of sufficiently ample divisor $D$ on $C$ will be disjoint 
from the vanishing loci of $\vp_1$ and $\vp_2$, and for every vector bundle 
$\sE|_x$ corresponding to a point of $\cU$, will satisfy 
$H^1(C,\sE|_x(D))=H^0(C,\sE|_x(-D))=0$. Fixing such a $D$, let
$D'=p_1^*(D)$ on $C\times \cM(2,\sL)$. Let $\cG$ be the 
relative Grassmannian $\cG(k,p_{2*} (\sE_{\cU}(D')/\sE_{\cU}(-D')))$. Since
$p_{2*} (\sE_{\cU}(D')/\sE_{\cU}(-D'))$ is locally free of rank $4 \deg D$, 
we have $\cG$ smooth of relative dimension $k(4 \deg D - k)$ over $\cU$. 
Let $\sF$ be the universal subbundle on $\cG$. As before, the open substack
$\cG_{\cU}(2,\sL,k) \subseteq \cG(2,\sL,k)$ parametrizing pairs for
which the underlying bundle corresponds to a morphism into $\cU$ is
realized as the closed substack of $\cG$ for which $\sF$
is contained in both $p_{2*}(\sE/\sE(-D'))$ and $p_{2*}\sE(D')$, which
have rank $2\deg D$ and $d+2\deg D +2-2g$ respectively.

Also as before, symplectic forms play a crucial role, but now we have
from Lemma \ref{lem:construct-form} that 
$p_{2*}(\sE(D')/\sE(-D'))$ carries two symplectic forms 
$\langle,\rangle_i$ induced by $\vp_i$ for $i=1,2$.
We thus have that $p_{2*}\sE(D')$ and
$p_{2*}(\sE/\sE(-D'))$ are both simultaneously isotropic for both
forms. Moreover, by Lemma \ref{lem:general-m}, we have that (by choosing
a more ample and more general $D$ as necessary) the space 
$\underline{\langle,\rangle}$ of forms spanned by the $\langle,\rangle_i$
is $2$-dimensional and consists entirely of forms with no degeneracy on any 
fiber of $p_{2*} \sE_{\cU}(D')$. Thus, we can
apply Corollary \ref{cor:iso-intersect-double} to find
that every component of $\cG_{\cU}(2,\sL,k)$ has codimension in $\cG$ bounded 
above by
$$k(4\deg D -d-2+2g)-2\binom{k}{2}.$$
By Corollary \ref{cor:smooth-codim}, we thus find that every component of 
$\cG_{\cU}(2,\sL,k)$ has dimension at least
$$3g-3-k(k-d+2g-2)+2\binom{k}{2}=\rho-g+2\binom{k}{2}.$$

As ${\cU}$ varies over all bounded open substacks of $\cM(2,\sL)$, it covers
the entire moduli space, so we conclude that the obtained dimension bounds
apply to all of $\cG(2,\sL,k)$, as desired.
\end{proof}

\section{Further discussion}

It is well known (see for instance Comments 4.6 of \cite{g-t1})
that in higher rank, there are cases where the expected 
dimension is positive but the Brill-Noether loci are empty. The same
is true for loci with fixed determinant and the appropriate modified
expected dimension. In particular, a component of the Brill-Noether locus
may be large enough to dominate the Jacobian, but may fail to do so. This 
makes it difficult to predict when a modified expected dimension in fixed 
determinant will in fact give rise to an entire component of the 
Brill-Noether locus of varying determinant. We conclude by examining this 
issue in detail in two specific examples.

\begin{ex} We consider the case $k=2$, and $d=g-2$, with $C$ assumed to
be general. The expected dimension of $\cG_C(2,g-2,2)$ is 
$\rho(2,g-2,2,g)-1=2g-8$ (the reduction by $1$ in dimension arised because
we are considering the stack rather than the coarse moduli space), so for 
$g \geq 7$, we have that $\cG_C(2,g-2,2)$
has dimension large enough to surject onto $\Pic^{g-2}(C)$; in particular,
for any line bundle $\sL$ of degree $g-2$, even the naive expected dimension 
of $\cG(2,\sL,2)$ is non-negative. Moreover, we have $h^1(C,\sL)>0$, so if 
$\cG_C(2,g-2,2)$ does in fact dominate $\Pic^{g-2}(C)$, it follows from
Theorem \ref{thm:exp-dim-twist} that it must have strictly larger than the
expected dimension.

However, at least on the stable locus $\cG_C^s(2,g-2,2)$, this is not the
case. Indeed, Teixidor \cite{te4} analyzes the case $r=k=2$ more generally, 
and shows that on a general curve, $\cG_C^s(2,d,2)$ is irreducible
of the expected dimension. She shows that the general linear series in
this space is of the form $(\sE,H^0(C,\sE))$, where $\sE$ is obtained as
an extension
$$0 \to \sO_C \to \sE \to \sL \to 0,$$
and we then necessarily have $h^0(C,\sL)=h^0(C,\det \sE)>0$. In our case
of degree $g-2$, such $\sL$ have codimension $2$ in $\Pic^{g-2}(C)$.
It follows that $\cG_C^s(2,g-2,2)$ does not dominate $\Pic^{g-2}(C)$, and
moreover that for a general $\sL \in \Pic^{g-2}(C)$ with $h^0(C,\sL)>0$,
we have $\dim \cG^s(2,\sL,2)>\rho^1_{\sL}(2,g)$. This is thus consistent
with Theorem \ref{thm:exp-dim-double}.
\end{ex}

However, we see that our techniques, and in particular the last statement
of Theorem \ref{thm:exp-dim-twist}, can shed some light on such emptiness
examples. Although the following corollary is rather trivial (and is in
any case an immediate consequence of Teixidor's analysis), it may be of
interest that the argument is quite atypical.

\begin{cor}\label{cor:empty} Suppose $\ch F =0$. Set $d=g-1-\ell$, with 
$C$ general and $\ell \geq 0$, and fix $\sL$ a line bundle
of degree $d$ with $h^0(C,\sL)=0$. Then there does not exist any semistable
vector bundle $\sE$ of determinant $\sL$ with $h^0(C,\sE) \geq 2$.
\end{cor}

\begin{proof} We begin with the case $k=2$, $d=g-1$, where $h^0(C,\sL)=0$
is equivalent to $h^1(C,\sL)=0$. But since $h^1(C,\sL(-P))>0$ for any 
$P \in C$, we have 
$\delta=1$, and we can choose $\Delta=P$ for any $P$. It follows from 
the final statement of Theorem \ref{thm:exp-dim-twist} that for any 
pair $(\sE,V)$ with determinant $\sL$, and for any $P \in C$, at least one 
section of $V$ vanishes at $P$. In particular, the sections of $V$ do not
generate $\sE$ at any point, and it follows that $V$ is 
contained in some line subbundle $\sM$ of $\sE$, which must then have
degree at least $\frac{g+2}{2}$, and therefore destabilizes $\sE$.

For the general case, $h^0(C,\sL)=0$ is equivalent to 
$h^0(C,\omega \otimes \sL^{-1})=h^1(C,\sL)=\ell$. Because we are
in characteristic $0$, for a general $P \in C$ the vanishing sequence for 
$h^0(C,\omega \otimes \sL^{-1})$ is generic, so for 
$\ell' \leq \ell$ we have that
$$h^1(C,\sL(\ell' P))=h^0(C,\omega \otimes \sL^{-1}(-\ell' P))=\ell-\ell',$$
and we conclude that $h^0(C,\sL(\ell' P))=0$. Thus, if we have 
$\sE$ of determinant $\sL$
with $h^0(C,\sE) \geq k$, and we set $\ell'=2 \lfloor \frac{\ell}{2}\rfloor$,
then for general $P$ we still have $h^0(C,\sE(\frac{\ell'}{2} P)) \geq k$ 
and $h^0(C,\sL(\ell' P))=0$. We claim that $\sE(\frac{\ell'}{2} P)$, and
hence $\sE$, is unstable. If $\ell$ is even, this is immediate from the
$d=g-1$ case above.  If $\ell$ is odd, then $\sE(\frac{\ell'}{2} P)$ has 
degree $g-2$. Let $\sE'$ be an inverse elementary transformation at a 
general point $Q$, so that $\det \sE' \cong \sL(\ell'P +Q)$ has degree $g-1$, 
and 
$h^0(C,\sL(\ell'P+Q))=0$ as before. By the above, $\sE'$ has a line subbundle
of degree at least $\frac{g+2}{2}$, so we conclude that 
$\sE(\frac{\ell'}{2}P)$ has a line subbundle of degree at least $\frac{g}{2}$,
and is therefore still unstable, as desired.
\end{proof}

We conclude by remarking that while we believe our main results show 
substantial
promise for generalization, possibly as far as a comprehensive collection of
expected dimensions in rank $2$, any generalization will not be routine.
One indication is that our results on expected dimensions thus far have not
required any stability hypotheses. However, if, as expected, one obtains
higher and higher expected dimensions as $h^1(C,\sL)$ grows, then these
modified estimates will almost certainly not apply to arbitrarily unstable
bundles. Indeed, for sufficiently unstable bundles, the determinant is
completely unrelated to the number of sections, so we should expect any
fully general answer to combine determinant conditions with stability
conditions.

\appendix
\section{Codimension in stacks}\label{sec:codim-stacks}

We give a basic treatment of codimension for algebraic stacks. We follow 
the terminology and conventions of \cite{l-m-b}. We assume throughout that
$\cX$ is a locally Noetherian algebraic stack.

\begin{defn} Let $\cX$ be a locally Noetherian 
algebraic stack, and $\cZ \subseteq \cX$ a 
closed substack. The {\bf codimension} of $\cZ$ in $\cX$ is defined to be 
the codimension in $X$ of $X \times_{\cX} \cZ$, where $X \to \cX$ is a 
smooth cover of $\cX$.
\end{defn}

\begin{prop}\label{prop:codim-defd} Codimension is well defined, and finite.
\end{prop}

\begin{proof} We first verify that codimension is well defined. If $X,X'$ 
are both smooth covers of $\cX$, then 
$X \times_{\cX} X'$ is a smooth cover of both $X$ and $X'$, so it is
enough to see that codimension is preserved under passage to a smooth 
cover. But this is true more generally for flat covers, see
Corollary 6.1.4 of \cite{ega42}.

We next verify finiteness. By definition, any smooth presentation 
$P:X \to \cX$ has
$X$ locally Noetherian. Now, $\codim_{\cX} \cZ := \codim_X P^{-1} Z$,
and in order to establish finiteness of $\codim_X P^{-1} Z$, it is enough
to show that $\codim_X Z'$ is finite for some irreducible component $Z'$
of $P^{-1} Z$. But let $Z'$ be any irreducible component; since $X$ is
locally Noetherian, there exists an open subscheme $U \subseteq X$ which
meets $Z'$ and is Noetherian. Then $\codim_X Z'= \codim_U Z'$, and
$\codim_U Z'$ is finite, as desired.
\end{proof}

\begin{lem}\label{lem:irred-cover} Let $\cZ \subseteq \cX$ be an 
irreducible component.
If $X \to \cX$ is a smooth morphism from a scheme $X$, then every
irreducible component of $Z:=\cZ \times_{\cX} X$ is an irreducible
component of $X$.
\end{lem}

\begin{proof} Let $\zeta$ the generic point of an irreducible component of 
$Z:=\cZ \times_{\cX} X$, and $\xi$ be the generic point of $\cZ$
(for the existence and uniqueness of generic points of algebraic stacks, 
see Corollary 5.7.2
of \cite{l-m-b}). Because $Z$ is flat over $\cZ$, it follows from
Proposition 5.6 and Corollary 5.7.1 of \cite{l-m-b} that the image
of $\zeta$ is $\xi$. But by the same argument, any generization of
$\zeta$ in $X$ must map to $\xi$ as well, in which case it is 
contained in $Z$ by definition. We conclude that $\zeta$ is a generic
point of $X$, as desired.
\end{proof}

We immediately conclude:

\begin{cor}\label{cor:irred-codim} Let $\cZ \subseteq \cX$ be an 
irreducible component. Then $\cZ$ has codimension $0$ in $\cX$.
\end{cor}

We now specialize to the case of stacks of finite type over a field.

\begin{prop}\label{prop:dim-defined} Suppose $\cX$ is an irreducible
algebraic stack, of finite type over a field. Then for any smooth
cover $P:X \to \cX$ and any irreducible component $Z$ of $X$, we have
$$\dim \cX = \dim Z-\dim_Z P,$$
where $\dim_Z P$ denotes the relative dimension of $P$ along $Z$.
\end{prop} 

Note that because the relative dimension of a smooth morphism is locally 
constant, $\dim _Z P$ is well defined. 

\begin{proof} Let $\eta$ be the generic point of $X$. Recall that every
irreducible component of $X$ dominates $\cX$,
and in particular if $Z'$ is any irreducible component meeting $Z$, and
$\xi, \xi'$ are the generic points $Z,Z'$ respectively, then using that
$X$ is locally of finite type over a field, the definitions give us
$$\dim Z - \dim_Z P= \dim_{\xi} Z-\dim_{\xi} P=:\dim_{\eta} \cX 
:=\dim_{\xi'} Z' - \dim_{\xi'} P = \dim Z' - \dim_{Z'} P.$$
But again using that $\dim P$ is locally constant, since $Z'$ intersects
$Z$ we conclude that $\dim Z = \dim Z'$.

Now suppose that $x \in P(Z)$, and
choose a point $z \in Z$ mapping to $x$. We then have
$$\dim_\eta \cX := \dim_{\xi} X- \dim_{\xi} P= \dim Z - \dim_Z P
= \dim_z X-\dim_z P =: \dim_x \cX,$$
where the identity $\dim Z = \dim_z X$ follows from the fact that any
irreducible component of $X$ meeting $Z$ has the same dimension as $Z$.

We thus have that $\dim_{\eta} \cX=\dim Z - \dim_Z P$, and further
that $\dim_{\eta} \cX=\dim_x X$ for any $x \in P(Z)$. However, since
$Z$ was an arbitrary component of $X$, we conclude that 
$\dim_{\eta} \cX = \dim_x X$ for every $x \in X$, and thus we obtain
the desired statement on $\dim X$.
\end{proof}

\begin{prop}\label{prop:codim-defined} Suppose $\cX$ is an irreducible
algebraic stack, of finite type over a field, and $\cZ \subseteq \cX$ is
an irreducible closed substack. Then for any smooth
cover $P:X \to \cX$ and any irreducible components $X'$ of $X$ and $Z$ of 
$P^{-1} \cZ$ with $Z \subseteq X'$, we have
$$\codim_{\cX} \cZ = \codim_{X'} Z = \dim \cX-\dim \cZ.$$
\end{prop} 

\begin{proof} By definition, we have $\codim_{\cX} \cZ=\codim_X P^{-1} \cZ$,
so the first identity amounts to showing that $\codim_{X'} Z$ is independent
of the choices of $X'$ and $Z$, which in turn is a consequence of the second
identity. Since $X'$ and $Z$ are irreducible and
$X$ is locally of finite type over a field, we have
$\codim_{X'} Z = \dim X' - \dim Z$. But according to Proposition 
\ref{prop:dim-defined}, we have 
$$\dim X' - \dim Z=\dim \cX+\dim_{X'} P -\dim \cZ - \dim_{Z} P=
\dim \cX-\dim \cZ,$$
where $\dim_{X'} P =\dim_Z P$ because $Z \subseteq X'$. 
We thus conclude the desired statement.
\end{proof}

\begin{cor}\label{cor:codim-equiv} Suppose $\cX$ is an 
algebraic stack of finite type over a field, and $\cZ \subseteq \cX$ is
a closed substack. Then 
$$\codim_{\cX} \cZ = \min_{\cZ' \subseteq \cZ} 
\left(\max_{\cX' \subseteq \cX: \cZ' \subseteq \cX'}
\codim_{\cX'} \cZ'\right),$$
where $\cZ'$ and $\cX'$ range over irreducible components of $\cZ$ and $\cX$
respectively.
\end{cor} 

\begin{proof} Both sides can be defined in terms of a smooth cover.
The only obstruction to the desired identity is that \textit{a priori} on 
the righthand side we have a minimum over a maximum over a minimum over a 
maximum, and we need to know that this may be rewritten as a single minimum 
over a maximum. But according to Proposition \ref{prop:codim-defined}, the 
codimension value is independent of the choices of components under the 
smooth cover, so we obtain the desired statement.
\end{proof}

\begin{cor}\label{cor:smooth-dim} Suppose that $B$ is an equidimensional 
scheme of finite type over a field, and $\cX$ is an algebraic stack smooth of
relative dimension $n$ over $B$. Then every irreducible component of $\cX$
has dimension $\dim B+n$.
\end{cor}

\begin{proof} Let $P:X \to \cX$ be a smooth presentation, and $x \in \cX$.
Choose any $\tilde{x} \in X$ lying over $x$, and let $m$ be the relative
dimension of $P$ at $\tilde{x}$. Then we have by definition
$\dim_x \cX=\dim_{\tilde{x}} X - m$,
and since $X$ is smooth over $B$ at $\tilde{x}$ of relative dimension $m+n$,
we have $\dim_{\tilde{x}} X=\dim B +m+n$, so we conclude that
$\dim_x \cX=\dim B +n$. Since $x$ was arbitrary, we have that every
component of $\cX$ has the desired dimension.
\end{proof}

Putting together Proposition \ref{prop:codim-defined}, Corollary
\ref{cor:codim-equiv}, and Corollary \ref{cor:smooth-dim}, we finally
conclude:

\begin{cor}\label{cor:smooth-codim} Suppose that $B$ is an equidimensional 
scheme of finite type over a field, $\cX$ is an algebraic stack smooth of
relative dimension $n$ over $B$, and $\cZ \subseteq \cX$ is an
irreducible closed substack, of codimension $c$. Then 
$$\dim \cZ = \dim B+ n-c.$$
\end{cor}

\bibliographystyle{hamsplain}
\bibliography{hgen}

\newcommand{\noopsort}[1]{} \newcommand{\printfirst}[2]{#1}
  \newcommand{\singleletter}[1]{#1} \newcommand{\switchargs}[2]{#2#1}
\providecommand{\bysame}{\leavevmode\hbox to3em{\hrulefill}\thinspace}
\begin{thebibliography}{1}

\bibitem{b-f2}
Aaron Bertram and Burt Feinberg, \emph{On stable rank two bundles with
  canonical determinant and many sections}, Algebraic Geometry (Catania
  1993/Barcelona 1994), Lecture Notes in Pure and Applied Mathematics, vol.
  200, Dekker, 1998, pp.~259--269.

\bibitem{ega42}
Alexander Grothendieck and Jean Dieudonn\'e, \emph{{\'E}l\'ements de
  g\'eom\'etrie alg\'ebrique: {IV.} \'{E}tude locale des sch\'emas et des
  morphismes de sch\'emas, seconde partie}, vol.~24, Publications
  math\'ematiques de l'I.H.\'E.S., no.~2, Institut des Hautes \'Etudes
  Scientifiques, 1965.

\bibitem{ega44}
\bysame, \emph{{\'E}l\'ements de g\'eom\'etrie alg\'ebrique: {IV.} \'{E}tude
  locale des sch\'emas et des morphismes de sch\'emas, quatri\'eme partie},
  vol.~32, Publications math\'ematiques de l'I.H.\'E.S., no.~2, Institut des
  Hautes \'Etudes Scientifiques, 1967.

\bibitem{g-t1}
Ivona Grzegorczyk and Montserrat {Teixidor i Bigas}, \emph{Brill-{N}oether
  theory for stable vector bundles}, Moduli Spaces and Vector Bundles (Leticia
  Brambila-Paz, Steven Bradlow, Oscar Garc\'ia-Prada, and S.\ Ramanan, eds.),
  London Mathematical Society Lectures Notes Series, vol. 359, Cambridge
  University Press, 2009, \mbox{arXiv:0801:4740}, pp.~29--50.

\bibitem{o-h1}
David Helm and Brian Osserman, \emph{Flatness of the linked {G}rassmannian},
  Proceedings of the AMS \textbf{136} (2008), no.~10, 3383--3390,
  \mbox{arXiv:math.AG/0605373}.

\bibitem{ho1}
Melvin Hochster, \emph{Big {C}ohen-{M}acaulay modules and algebras and
  embeddability in rings of {W}itt vectors}, Queen's Papers on Pure and Applied
  Math \textbf{42} (1975), 106--195.

\bibitem{l-m-b}
G\'erard Laumon and Laurent Moret-Bailly, \emph{Champs algebriques},
  Springer-Verlag, 2000.

\bibitem{mu2}
Shigeru Mukai, \emph{Vector bundles and {B}rill-{N}oether theory}, Current
  topics in complex algebraic geometry, MSRI Publications, vol.~28, Cambridge
  University Press, 1995, pp.~145--158.

\bibitem{te4}
Montserrat {Teixidor i Bigas}, \emph{Brill-{N}oether theory for vector bundles
  of rank 2}, Tohoku Mathematical Journal \textbf{43} (1991), no.~1, 123--126.

\end{thebibliography}

\end{document}